\documentclass[12pt]{amsart}
\usepackage{fullpage}
\usepackage[all]{xy}
\usepackage{amscd, amssymb, amsmath, amsthm, graphics}
\usepackage{amsmath,amsfonts,amsthm,amssymb}
\usepackage{latexsym,amsmath}
\usepackage{graphicx,psfrag}
\usepackage{pinlabel}
\usepackage{mathrsfs}
\usepackage{xcolor}
\usepackage{hyperref}
\usepackage{lmodern}
\usepackage[T1]{fontenc}

\newtheorem{theorem}{Theorem}[section]

\newtheorem{corollary}[theorem]{Corollary}
\newtheorem{lemma}[theorem]{Lemma}

\theoremstyle{definition}
\newtheorem{definition}[theorem]{Definition}

\newtheorem{problem}[theorem]{Problem}

\theoremstyle{remark}
\newtheorem{remark}[theorem]{Remark}

\numberwithin{equation}{section}
 \newcommand{\bZ}{\mathbb Z}
 \newcommand{\bR}{\mathbb R}
 \newcommand{\bN}{\mathbb N}
 \newcommand{\bC}{\mathbb C}
 
 \newcommand{\bH}{\mathbb H}

  \makeatletter
\@namedef{subjclassname@2020}{\textup{2020} Mathematics Subject Classification}
\makeatother
  \dedicatory{Dedicated to Shicheng Wang on the occasion of his 70th birthday}
  
\begin{document}
\sloppy

\title[On $\pi_1$-injectivity of self-maps in low dimensions]
{On $\pi_1$-injectivity of self-maps in low dimensions}

\author{Christoforos Neofytidis }
\address{Department of Mathematics and Statistics, University of Cyprus, Nicosia 1678, Cyprus}
\email{neofytidis.christoforos@ucy.ac.cy}


\date{\today}

\begin{abstract}
We show that all self-maps of non-zero degree of $3$-manifolds not covered by $S^3$ and of Thurston geometric $4$-manifolds and their connected sums not diffeomorphic to  $N\#(\#_{p\geq0}S^2\times S^2)\#(\#_{q\geq0}\bC P^2)$, where $N$ is an $S^2\times\mathbb X^2$ or $S^3\times\mathbb R$ manifold, are $\pi_1$-injective. We thus determine when these maps induce $\pi_1$-isomorphisms. The results in dimension three were previously established by Shicheng Wang. We give a uniform group theoretic proof in all cases based only on the residual finiteness of the fundamental groups for the $\pi_1$-injectivity and then only on numerical invariants for the $\pi_1$-isomorphisms. 
\end{abstract}

\maketitle
\vspace{-.5cm}

\section{Introduction}

A map of non-zero degree $f\colon M\to N$, where $M$ and $N$ are closed oriented connected $n$-manifolds, induces a homomorphism $f_*\colon \pi_1(M)\to \pi_1(N)$ which has image of finite index in $\pi_1(N)$. In particular, when the degree is $\deg(f)=\pm 1$,  the homomorphism $f_*$ is surjective, and we say that $f$ is {\em $\pi_1$-surjective}. In the case of self-maps, i.e., when $M=N$, an open question attributed to H. Hopf asks whether the assumption of having degree $\pm 1$ implies that $f$ is a homotopy equivalence:

\begin{problem}[Hopf]\cite[Problem 5.26]{Kir}\label{p:Hopf}
Let $M$ be a closed oriented connected manifold. Is every self-map $f\colon M\to M$ with $|\deg(f)|=1$ a homotopy equivalence?
\end{problem}

Clearly, a first step in understanding Hopf's problem is to determine whether $f$ is {\em $\pi_1$-injective} whenever $\deg(f)=\pm1$.  In dimension three, a notable result of S. Wang~\cite{Wa} says that $\pi_1$-injectivity holds for self-maps of non-zero degree of most $3$-manifolds without the strong constraint of having degree $\pm1$:

\begin{theorem}\cite[Theorem 0.1]{Wa}\label{t:Wang}
Suppose $M$ is a closed oriented connected $3$-manifold which is not covered by $S^3$. Then any non-zero degree map $f\colon M\to M$ is $\pi_1$-injective.
\end{theorem}

For $3$-manifolds with non-vanishing hyperbolic (simplicial) volume, Wang's proof relied on the residual finiteness of fundamental groups of $3$-manifolds. For the remaining cases, Wang provided a case-by-case analysis essentially by following Thurston's geometrisation picture.  One of the goals of this paper is to indicate that residual finiteness alone suffices in all cases of Theorem \ref{t:Wang}. We will prove the following result which encompasses Theorem \ref{t:Wang}, as well as the analogous picture in four dimensions:

\begin{theorem}\label{t:main}
Let $M$ be a closed oriented connected manifold which belongs to one of the following two classes
\begin{itemize}
\item[(i)] all $3$-manifolds that are not covered by $S^3$;
\item[(ii)] all Thurston geometric $4$-manifolds and their connected sums that are not diffeomorphic to  $N\#(\#_{p\geq0}S^2\times S^2)\#(\#_{q\geq0}\pm\bC P^2)$, where $N$ is modeled on one of the geometries $S^3\times \mathbb R$ or $S^2\times\mathbb X^2$.
\end{itemize}
Then any non-zero degree map $f\colon M\to M$ is $\pi_1$-injective.
\end{theorem}

Note that the phrase ``not covered by $S^3$'' in part (i) of the above theorem is equivalent to saying ``not modeled to $S^3$'', following Thurston's geometrization picture. Also, by $\pm\bC P^2$ we mean the two orientations on the complex projective plane (the number of summands $q$ could be only for one of the two or a mixture of the two). Finally, $\mathbb X^2$ denotes the three 2-dimensional geometries, see Section \ref{s:geometric}. Similarly, $\mathbb X^3$ below (for instance in Corollary \ref{c:main}) will denote any 3-dimensional geometry.

Subsequently to Theorem \ref{t:Wang}, Wang determined when actually the maps of Theorem \ref{t:Wang} are $\pi_1$-isomorphisms:

\begin{corollary}\cite[Corollary 0.2]{Wa}\label{c2:Wang}
Suppose $M$ is a closed oriented connected $3$-manifold. Then any non-zero degree map $f\colon M\to M$ induces a $\pi_1$-isomorphism unless $M$ is covered by either a $T^2$-bundle over $S^1$, or $\Sigma\times S^1$ for some surface $\Sigma$, or $S^3$.
\end{corollary}

In terms of Thurston geometries, Corollary \ref{c2:Wang} says roughly that self-maps of non-zero degree induce $\pi_1$-isomorphisms always on non-geometric $3$-manifolds and sometimes on geometric ones. More precisely,  Corollary \ref{c2:Wang} says that any self-map of non-zero degree of a $3$-manifold induces a $\pi_1$-isomorphism unless $M$ is modeled on one of the geometries $\bR^3$, $Nil^3$, $Sol^3$, $\bH^2\times\bR$, $S^2\times\bR$ or $S^3$ (but clearly $M$ is not $S^3$ itself). Wang's proof in~\cite{Wa} relies on the non-vanishing of the simplicial volume for hyperbolic $3$-manifolds as well as for $3$-manifolds with a hyperbolic piece in their JSJ decomposition, and then proceeds again with a case-by-case study, including a careful analysis of endomorphisms of non-trivial free products, since the fundamental group of a connected sum is the free product of the fundamental groups of the summands. Having established Theorem \ref{t:main}, we will see that actually the non-vanishing of various numerical invariants alone suffices. Indeed, the following consequence of Theorem \ref{t:main}, about $3$- and $4$-manifolds, encompasses Corollary \ref{c2:Wang}:

\begin{corollary}\label{c:main}
Let $M$ be a closed oriented connected manifold which belongs to one of the following two classes
\begin{itemize}
\item[(i)] all $3$-manifolds which are not modeled on one of the geometries $\bR^3$, $Nil^3$, $Sol^3$, $\bH^2\times\bR$, $S^2\times\bR$ or $S^3$. 
\item[(ii)] all (possibly trivial) connected sums of geometric $4$-manifolds which are not diffeomorphic to 
$N\#(\#_{p\geq0}S^2\times S^2)\#(\#_{q\geq0}\pm\bC P^2)$, where $N$ is modeled on one of the geometries $Sol^4_0$, $Sol^4_1$, $Sol^4_{m\neq n}$, $Nil^4$, $\mathbb X^3\times \bR$, $S^2\times\mathbb X^2$ or is a connected sum of any two copies among the two oriented quotients of $S^2\times S^2$ whose fundamental group has order two.
\end{itemize}
Then any non-zero degree map $f\colon M\to M$ is a $\pi_1$-isomorphism.
\end{corollary}

\begin{remark}
In part (ii) of the above statement, one could include the three last connected sums of quotients of $S^2\times S^2$ in the manifold $N\#(\#_{p\geq0}S^2\times S^2)\#(\#_{q\geq0}\pm\bC P^2)$, by replacing ``not diffeomorphic to'' with ``not covered by''.
\end{remark}

In fact, group theoretic results stronger than Corollary \ref{c:main} will be proved in Section \ref{s:iso} (Theorems \ref{t:coHopf3} and \ref{t:coHopf4}). Also, as we already indicated in the above paragraph, we clearly do not take into account in Corollary \ref{c:main} (or in Theorem \ref{t:main}) some trivial cases: $S^3$ in the case of $3$-manifolds, and $S^4$, $\pm\mathbb{C}P^2$ or $S^k$-bundles ($k=2,3$) over oriented manifolds for the summand $N$ in the case of $4$-manifolds.

In~\cite{NeoHopf}, a strong version of Problem \ref{p:Hopf} was proposed for aspherical manifolds: 

\begin{problem}[Strong version of the Hopf problem for aspherical manifolds]\cite[Problem 1.2]{NeoHopf}\label{p:sHopf}
Let $M$ be a closed oriented connected aspherical manifold. Is every self-map $f\colon M\to M$ with $\deg(f)\neq0$ either a homotopy equivalence or homotopic to a non-trivial covering?
\end{problem}

The proofs of Theorem \ref{t:main} and Corollary \ref{c:main} involve for the most part the study of endomorphisms of fundamental groups of aspherical manifolds and their free products. Within this study, we will obtain the following partial answer to Problem \ref{p:sHopf} in any dimension: 

\begin{theorem}\label{t:hopf}
Let $M$ be a closed oriented connected aspherical manifold with Hopfian fundamental group, which has a finite cover $\overline M$ with $I(\overline M)\neq0$, where $I$ is a functorial numerical invariant\footnote{See Definition \ref{d:functorial}}. Then any self-map $g\colon M\to M$ of non-zero degree is a homotopy equivalence. In particular, $g$ lifts to a homotopy equivalence $\bar g\colon \widehat M\to \overline M$, where $\widehat M$ is a finite covering of $M$ determined by $g$ and $\overline M$.
\end{theorem}

What makes Theorem \ref{t:hopf} interesting is that the non-vanishing assumption $I(\overline M)\neq0$ does not necessarily hold for $M$, i.e. $I(M)$ could be zero. In addition, we remark that the assumption that $\pi_1(M)$ is Hopfian might be redundant, since no example of a closed aspherical manifold with non-Hopfian fundamental group is known.

\subsection*{Acknowledgments}
I would like to thank Grigori Avramidi for useful discussions, as well as an anonymous referee for thoroughly reading this paper and for all the comments. Parts of this project were carried out during research stays at Max Planck Institute for Mathematics in Bonn in 2024 and 2025. Partial support by a start-up fund from the University of Cyprus, as well as by the research grant EXCELLENCE/0925-Hopf problem and monotone invariants under the program “THALIA 2021-2027”, co-funded by the European Union through the Research and Innovation Foundation of Cyprus, is gratefully acknowledged.

\section{Geometric manifolds in low dimensions}\label{s:geometric}

An $n$-dimensional (maximal) geometry in the sense of Thurston is a complete simply connected Riemannian manifold $\mathbb{X}^n$ on which $\mathrm{Isom}(\mathbb{X}^n)$ acts transitively, contains discrete subgroups acting freely on $\mathbb{X}^n$ with finite-volume quotients, and is maximal among Lie groups acting smoothly and effectively on $\mathbb{X}^n$ with compact stabilizers. We say that a manifold $M$ is an {\em $\mathbb{X}^n$ manifold}, or is {\em modeled on $\mathbb{X}^n$}, or {\em carries the $\mathbb{X}^n$ geometry} in the sense of Thurston, if it is diffeomorphic to a quotient of
$\mathbb{X}^n$ by a lattice $G$ in the group of isometries of $\mathbb{X}^n$, where $G=\pi_1(M)$. 
We say that two geometries $\mathbb{X}^n$ and $\mathbb{Y}^n$ are the same if there is a diffeomorphism $\psi \colon \mathbb{X}^n
\longrightarrow \mathbb{Y}^n$ and an isomorphism $\mathrm{Isom}(\mathbb{X}^n) \longrightarrow \mathrm{Isom}(\mathbb{Y}^n)$ that sends each $g \in \mathrm{Isom}(\mathbb{X}^n)$ to $\psi \circ g \circ \psi^{-1} \in \mathrm{Isom}(\mathbb{Y}^n)$.

Below, we describe all closed oriented low dimensional manifolds that carry a Thurston geometry. 

\subsection*{Dimensions one and two}

The circle $S^1=\bR/\bZ$ is the only closed $1$-manifold; it is modeled on $\bR$. In dimension two, a closed oriented surface of genus $g\geq0$ will be denoted by $\Sigma_g$: For $g=0$ we have the 2-sphere $\Sigma_0=S^2$, which is modeled on $S^2$, for $g=1$ the 2-torus $\Sigma_1=T^2=\bR^2/\bZ^2$, which is modeled on $\bR^2$, and for $g\geq2$ hyperbolic surfaces $\Sigma_g=\mathbb{H}^2/\pi_1(\Sigma_g)$, which are modeled on $\mathbb{H}^2$, where
\[
\pi_1(\Sigma_g)=\langle a_1,b_1,...,a_g,b_g \ | \ [a_1,b_1]\cdots[a_g,b_g]=1\rangle.
\]
Table \ref{table:2geom} summarises the geometries in dimension two.
\begin{table}[!ht]
\centering
{\small
\begin{tabular}{c|c}
Type & Geometry $\mathbb{X}^2$\\
\hline
         Spherical   & $S^2$\\
         Euclidean    & $\bR^2$ \\
Hyperbolic & $\mathbb{H}^2$\\            
\end{tabular}}
\vspace{9pt}
\caption{{\small The $2$-dimensional Thurston geometries}}\label{table:2geom}
\end{table}

\subsection*{Dimension three}

Thurston proved that there are eight geometries up to equivalence, and they are homotopically unique, i.e., two closed geometric manifolds which are homotopy equivalent are modelled on the same geometry: 
$$
\mathbb{H}^3, Sol^3,
\widetilde{SL_2}, \mathbb{H}^2 \times \bR, Nil^3, \bR^3, S^2 \times \bR, S^3.
$$
In Table \ref{table:3geom}, we list the finite covers for closed manifolds in each of those geometries and refer to~\cite{Thu,Scott,Agol} for the proofs. 

\begin{table}[!ht]
\centering
{\small
\begin{tabular}{r|l}
Geometry $\mathbb{X}^3$ & $M$ is finitely covered by...\\
\hline
$\mathbb{H}^3$     & a mapping torus of a hyperbolic surface with pseudo-Anosov monodromy\\
$Sol^3$            & a mapping torus of $T^2$ with hyperbolic monodromy\\
$\widetilde{SL_2}$ & a non-trivial $S^1$-bundle over a hyperbolic surface\\  
$Nil^3$            & a non-trivial $S^1$-bundle over $T^2$\\
$\mathbb{H}^2 \times \bR$ & a product of $S^1$ with a hyperbolic surface\\
$\bR^3$             & the $3$-torus $T^3$\\
$S^2 \times \bR$    &  the product $S^2 \times S^1$\\
$S^3$              & the $3$-sphere $S^3$
\end{tabular}}
\vspace{9pt}
\caption{{\small Finite covers of Thurston geometric 3-manifolds.}}\label{table:3geom}
\end{table}

\subsection*{Dimension four}
The $4$-dimensional Thurston's geometries were classified by Filipkiewicz in his thesis~\cite{Filipkiewicz}. In Table \ref{table:4geom}, we list the geometries that are realised by compact manifolds, following~\cite{Wall1,Wall2} and~\cite{Hillman}. In Table \ref{table:4geomvirtual} we list finite covers for manifolds that are not modeled on a hyperbolic geometry or the irreducible $\mathbb H^2\times \mathbb H^2$ geometry; we refer to~\cite{Hillman} for most of the proofs (see also~\cite[Section 6]{NeIIPP} on how one can distinguish further some of the nilpotent or solvable geometries which appear to have finite covers of the same type in Table \ref{table:4geomvirtual}). Finally, note that the geometries $S^4$ and $\mathbb{CP}^2$ do not have non-trivial orientable geometric quotients.

\begin{table}[!ht]
\centering
{\small
\begin{tabular}{r|l}
Type  & Geometry $\mathbb{X}^4$\\
\hline
Hyperbolic & $\mathbb{H}^4$, $\mathbb{H}^2(\mathbb{C})$\\        
                Solvable  & $Nil^4$, 
$Sol^4_{m \neq n}$, $Sol^4_0$,
             $Sol^4_1$, $Sol^3\times\bR$, $Nil^3\times\mathbb{R}$,  $\bR^4$\\
             Compact &  $S^4$, $\mathbb{CP}^2$, $S^2\times S^2$\\
                 Mixed  products  &  $S^2\times\mathbb{H}^2$, $S^2\times \bR^2$,  $S^3 \times \bR$, 
 $\mathbb{H}^3\times\mathbb{R}$,
           $\mathbb{H}^2\times\mathbb{R}^2$, $\mathbb{H}^2\times\mathbb{H}^2$,
              $\widetilde{SL_2}\times\mathbb{R}$ \\
\end{tabular}}
\vspace{9pt}
\caption{{\small The 4-dimensional Thurston geometries with compact representatives.}}\label{table:4geom}
\end{table}

\begin{table}[!ht]
\centering
{\small
\begin{tabular}{r|l}
Geometry $\mathbb{X}^4$ & $M$ is finitely covered by...\\
\hline
\text{reducible} \ $\mathbb{H}^2\times\mathbb{H}^2$     & a product of two hyperbolic surfaces\\
$S^2\times\mathbb{H}^2$ & an $S^2$-bundle over a hyperbolic surface\\
$\mathbb{X}^3\times\mathbb R$ & a product of an $\mathbb{X}^3$-manifold with $S^1$\\
$Sol^4_0$            & a non-trivial mapping torus of $T^3$\\
$Sol^4_1$            & a non-trivial mapping torus of a closed $Nil^3$-manifold\\
$Sol^4_{m\neq n}$            & a non-trivial mapping torus of $T^3$\\
$Nil^4$            &   a non-trivial mapping torus of a closed $Nil^3$-manifold\\
$S^2 \times S^2$    &  the product $S^2 \times S^2$\\
$\mathbb{CP}^2$             & the complex projective spaces $\pm\mathbb{C}P^2$  \\
$S^4$              & the $4$-sphere $S^4$ \\
\end{tabular}}
\vspace{9pt}
\caption{{\small Finite covers of Thurston geometric 4-manifolds that are not hyperbolic or irreducible $\mathbb H^2\times \mathbb H^2$ manifolds.}}\label{table:4geomvirtual}
\end{table}

\section{Residual finiteness}

Recall that a group $G$ is called {\em residually finite} if for any non-trivial element $g\in G$ there is a finite group $A$ and a homomorphism $\varrho\colon G\to A$ such that $\varrho(g)\neq1$. Equivalently, $G$ is residually finite if the intersection of all its (normal) subgroups of finite index is trivial. Also, we say that a group $G$ is {\em Hopfian} or it has the {\em Hopf property} if every surjective endomorphism of $G$ is an automorphism.

\subsection{Hirshon's work} 
By a theorem of Mal'cev~\cite{Mal}, every finitely generated residually finite group is Hopfian. Hirshon~\cite{Hir} proved the following  generalisation of Mal'cev's theorem:

\begin{theorem}[Hirshon]\label{t:Hirshon}
Let $G$ be a finitely generated residually finite group. If $\varphi$ is an endomorphism of $G$ such that $[G:\varphi(G)]<\infty$, then there is some $n\in\bN$ such that $\varphi$ is injective on $\varphi^n(G)$.
\end{theorem}

An interesting consequence of Theorem \ref{t:Hirshon} is the following, for which we include a proof for clarity:

\begin{corollary}[Hirshon]\label{c:Hirshon}
Let  $G$ be a finitely generated residually finite group. If $\varphi$ is an endomorphism of $G$ such that $[G:\varphi(G)]<\infty$, then $\ker(\varphi)$ is finite. In particular, $\varphi$ is injective if $G$ is either
\begin{itemize}
\item[(1)] torsion-free or 
\item[(2)] a non-trivial free product.
\end{itemize}
\end{corollary}
\begin{proof}
By Theorem \ref{t:Hirshon}, there is some $n\in\bN$ such that the restriction $\varphi|_{\varphi^n(G)}$ is injective, that is, $\ker(\varphi)\cap\varphi^n(G)=1$. Hence,
\[
|\ker(\varphi)|=[\ker(\varphi):\ker(\varphi)\cap\varphi^n(G)]\leq[G:\varphi^n(G)]<\infty
\]
as claimed.

Now, if $G$ is torsion-free, then clearly $\varphi$ is injective. If $G$ is a non-trivial free product, then the injectivity of $\varphi$ follows from the fact that $G$ does not contain any non-trivial finite normal subgroup.
\end{proof}

\subsection{Preservation of residual finiteness}

Despite the strong relationship between the Hopf property and residual finiteness, as evidenced in Mal'cev's and Hirshon's work, and the fact that the Hopf property is generally not preserved under subgroups or extensions, residual finiteness does respect many group theoretic operations. We gather here some of these results needed in our study.

First, we recall that residual finiteness passes both to finite-index subgroups and supergroups:

\begin{lemma}\label{l:sup-sub}
Let $G$ be a group and $H$ be a finite-index subgroup of $G$. Then $G$ is residually finite if and only if $H$ is residually finite.
\end{lemma}

The proof of the above lemma is left as an exercise. Note that the ``only if" direction does not require the assumption $[G:H]<\infty$.

The preservation of residual finiteness under extensions holds in very general forms, as shown  by Mal'cev and Miller:

\begin{theorem}\cite[Chapter III, Theorem 7]{Miller}\label{t:Mil}
Suppose $G$ fits into an extension $1\to N\to G\to Q\to 1$, where both $N$ and $Q$ are residually finite and $N$ is finitely generated. Moreover, assume that any of the following holds:
\begin{itemize}
\item[(1)] $N$ has trivial center;
\item[(2)] the extension splits; 
\item[(3)] $Q$ is free or $N$ is non-abelian free. 
\end{itemize}
Then $G$ is residually finite.
\end{theorem}

Finally, residual finiteness is preserved under taking free products as shown by Gruenberg and Baumslag:

\begin{theorem}\cite[Theorem 4.1]{Gru}\label{t:Gru}
The free product of residually finite groups is residually finite.
\end{theorem}

\section{Proof of Theorem \ref{t:main}}

We will now prove Theorem \ref{t:main}. Naturally, we will split the proof into the two dimensions three and four. 

\subsection{3-manifolds}

The starting point is the following theorem, which is based on a result of Hempel~\cite{He} and holds for all compact $3$-manifolds thanks to geometrization:

\begin{theorem}\label{t:He}
Let $M$ be a compact $3$-manifold. Then $\pi_1(M)$ is residually finite.
\end{theorem}

Roughly, Theorem \ref{t:He} follows from the residual finiteness of cyclic and surface groups together with the preservation of residual finiteness under finite-index supergroups (Lemma \ref{l:sup-sub}), extensions (Theorem \ref{t:Mil}), as well as free products (Theorem \ref{t:Gru}) and amalgamated free products over $\bZ^2$~\cite{He} of $3$-manifold groups.

\smallskip

Applying now Hirshon's work we can obtain the proof, which gives us in particular a quick way of showing Theorem \ref{t:Wang}: Let $f\colon M\to M$ be a map of non-zero degree, where $M$ is a closed oriented $3$-manifold. Then $[\pi_1(M):f_*(\pi_1(M))]<\infty$. Recall first that any prime $3$-manifold not modeled on $S^3$ is either aspherical or finitely covered by $S^2\times S^1$ and thus has torsion-free fundamental group. Hence, if the (possibly trivial) prime decomposition of $M$ contains no quotients of $S^3$, then $\pi_1(M)$ is torsion-free, hence $f_*$ is injective by Corollary \ref{c:Hirshon} (1). Finally, if $M$ has non-trivial prime decomposition with at least one summand modeled on $S^3$ (clearly except $S^3$), then $\pi_1(M)$ is a non-trivial free product, and thus $f_*$ is injective by Corollary \ref{c:Hirshon} (2).

This finishes the proof of the first part of Theorem \ref{t:main} (and reproves Theorem \ref{t:Wang}).

\subsection{4-manifolds}

First, we indicate that our $4$-manifolds in Theorem \ref{t:main} have residually finite fundamental groups:

\begin{theorem}\label{t:4dRF}
Let $M$ be a closed oriented geometric $4$-manifold or a connected sum of such $4$-manifolds. Then $\pi_1(M)$ is residually finite.
\end{theorem}
\begin{proof}
Suppose first that $M$ is a geometric $4$-manifold. 
Knowing the residual finiteness of the fundamental groups in lower dimensions (see Theorem \ref{t:He} and the comments below that theorem) allows us to apply Theorem \ref{t:Mil} and Lemma \ref{l:sup-sub} to the descriptions of finite coverings given in Table \ref{table:4geomvirtual} to conclude that if $M$ is not hyperbolic or an irreducible $\mathbb H^2\times\mathbb H^2$ manifold, then $\pi_1(M)$ is residually finite. For locally symmetric spaces of non-compact type see~\cite[Chapter III.$\Gamma$, Sec.7, pg. 512]{BH} (which uses linearity of their fundamental groups). Finally, suppose $M$ is a non-trivial connected sum of geometric $4$-manifolds. If $\pi_1(M)$ is freely indecomposable, then it is isomorphic to the fundamental group of one of the geometric summands of $M$, thus reducing to the previous case. Otherwise, we are in the situation of a non-trivial free product of residually finite groups and the claim follows by Theorem \ref{t:Gru}.
\end{proof}

We are now ready to prove Theorem \ref{t:main} in dimension four: Let $M$ be a closed oriented $4$-manifold which is either geometric or a connected sum of geometric $4$-manifolds. If $f\colon M\to M$ is a map of non-zero degree, then $[\pi_1(M):f_*(\pi_1(M))]<\infty$. Suppose first that $M$ has no $S^3\times\mathbb R$ or $S^2\times\mathbb X^2$ pieces in its (possibly trivial) connected sum decomposition. Then $\pi_1(M)$ is torsion-free, when the pieces with non-trivial fundamental group are only aspherical, or even trivial, hence Corollary \ref{c:Hirshon} (1) tells us that $f_*$ is injective. Next, suppose that $M$ contains an $S^3\times \mathbb R$  or an $S^2\times\mathbb X^2$ piece in its connected sum decomposition. 
If $M$ contains at least another piece in its connected sum decomposition with non-trivial fundamental group, 
then $f_*$ is injective by Corollary \ref{c:Hirshon} (2).

\medskip

This completes the proof of Theorem \ref{t:main}.

\section{From injectivity to isomorphisms}\label{s:iso}

In this section, we will show that many times we can pass from $\pi_1$-injective maps to $\pi_1$-isomorphisms, proving in particular Corollary \ref{c:main}. 

\subsection{Multiplicative and functorial invariants}

Having established the $\pi_1$-injectivity for self-maps as in Theorem \ref{t:main}, we would like now to examine whether one could say even more about these endomorphisms, namely, whether they are isomorphisms. Since $[\pi_1(M):f_*(\pi_1(M))]<\infty$ for any self-map of non-zero degree $f\colon M\to M$, the formalization of the latter problem is captured by the following notion:

\begin{definition}\label{d:finite-co-Hopf}
A group $G$ is called {\em finitely co-Hopfian} if every injective endomorphism onto a finite index subgroup of $G$ is an isomorphism.
\end{definition}

The finitely co-Hopf property is a weakening of the {\em co-Hopf property}, where only injectivity is assumed.

The following lemma tells us that the existence of non-zero multiplicative invariants under finite-index subgroups implies the finite co-Hopf property.

\begin{definition}\label{d:multiplicative}
A {\em multiplicative invariant with respect to finite-index subgroups} is a (finite) numerical group isomorphism invariant $\iota$ such that for any group $G$ and any finite-index subgroup $H\subseteq G$ it holds
$
\iota(H)=[G:H]\iota(G).
$
\end{definition}

\begin{lemma}\label{l:multiplicative}
Let $G$ be a group such that $\iota(G)$ is finite and not zero for some multiplicative invariant $\iota$. Then $G$ is finitely co-Hopfian.
\end{lemma}
\begin{proof}
Let $\varphi\colon G\to G$ be an injective endomorphism such that $[G:\varphi(G)]<\infty$. Then
\[
\varphi\colon G\to\varphi(G)
\]
is an isomorphism, hence
\[
\iota(G)=\iota(\varphi(G))=[G:\varphi(G)]\iota(G).
\]
Thus $[G:\varphi(G)]=1$ and the lemma follows.
\end{proof}

Under finiteness assumptions on a group $G$, two prominent examples of multiplicative invariants $\iota(G)$ are
\begin{itemize}
\item the Euler  characteristic $\chi(G)$ of $G$~\cite[Chapter IX, Proposition 7.3]{Br}, or
\item  any $L^2$-Betti number $b_*^{(2)}(G)$ of $G$~\cite[Theorem 4.15 (iii)]{Ka}. 
\end{itemize}
While the former example is more classical, its multiplicativity property follows as well by the latter, since the Euler characteristic is the alternating sum of the $L^2$-Betti numbers~\cite{Lu}.

The above group theoretic approach stems from the topological counterpart of covering spaces and, more generally, functorial invariants of manifolds: 

\begin{definition}\label{d:functorial}
A {\em functorial} invariant (for manifolds) is a nonnegative real-valued homotopy invariant $I$ such that 
for any map of non-zero degree $f\colon M\to N$ it holds $I(M)\geq|\deg(f)|I(N)$.
\end{definition}

We have the following immediate observation:

\begin{lemma}\label{l:functorial}
Let $I$ be a functorial invariant and $M$ be a closed manifold such that $I(M)$ is non-zero and finite. Then any non-zero degree self-map $f\colon M\to M$ is $\pi_1$-surjective.
\end{lemma}

Gromov introduced the notion of functorial invariants with the {\em simplicial volume} being probably among the most well-known examples of functorial invariants, see~\cite{Gr}. In dimension three, Brooks and Goldman introduced the {\em Seifert volume} (as an analogue to the simplicial volume) for $3$-manifolds modeled on the $\widetilde{SL_2}$ geometry~\cite{BG}.

\subsection{Proof of Corollary \ref{c:main} in dimension three}

We need to show that for each of the manifolds satisfying the hypothesis of Corollary \ref{c:main} (i) every self-map of non-zero degree is a $\pi_1$-isomorphism. We will show the following stronger result at the group level:

\begin{theorem}\label{t:coHopf3}
Let $M$ be a closed oriented connected $3$-manifold with non-trivial fundamental group, which is not modeled on one of the geometries $\bR^3$, $Nil^3$, $Sol^3$, $\bH^2\times\bR$ or $S^2\times\bR$. If $M$ is prime, then $\pi_1(M)$ is co-Hopfian. If $M$ is not prime, then $\pi_1(M)$ is finitely co-Hopfian.
\end{theorem}
\begin{proof}
Suppose first that $M$ is prime and let $\varphi\colon \pi_1(M)\to \pi_1(M)$ be an injective endomorphism. We may also assume that $M$ is not an $S^3$ manifold (otherwise $\pi_1(M)$ is obviously co-Hopfian), hence it is either modeled on one of the geometries  $\mathbb H^3$ or $\widetilde{SL_2}$, or it has non-trivial JSJ-decomposition. This means $M$ is aspherical. Hence, the self-map induced on the classifying space is given, up to homotopy, by $f:=B\varphi\colon M\to M$ and has finite degree $\deg(f)=\pm[\pi_1(M):\varphi(\pi_1(M))]$; see~\cite[Corollary 7.3]{NeEndo}. 

If $M$ is modeled on the $\mathbb H^3$ geometry or has a hyperbolic piece in its JSJ-decomposition, then the simplicial volume of $M$ does not vanish~\cite{Gr}.  If $M$ is modeled on the $\widetilde{SL_2}$ geometry,  then it has non-zero Seifert volume~\cite{BG}. Hence, in all these cases Lemma \ref{l:functorial} tells us that $\varphi$ is $\pi_1$-surjective as well. 

The only remaining case for $M$ prime is when $M$ has non-trivial JSJ-decomposition and all pieces are Seifert. In that case $M$ has zero simplicial volume and may have zero Seifert volume as well. However, Derbez and Wang~\cite{DW,DW1} proved that $M$ has always a finite covering $\overline M$ with non-zero Seifert volume. We can then show that $M$ has self-maps of absolute degree at most one along the lines of the proof of~\cite[Lemma 3.1]{DW}, with a stronger conclusion in our setting: Suppose $g\colon M\to M$ is a map of non-zero degree and let $p\colon \overline M\to M$ be a finite covering with $I(\overline M)\neq0$, where $I$ is the Seifert volume in the case under consideration. Let $q\colon \overline M_g\to M$ be the finite covering corresponding to $g_*^{-1}p_*(\pi_1(\overline M))$ and $\bar g\colon \overline M_g\to \overline M$ be a lift of $g\circ q$.  This is shown in the following commutative diagram.
$$
\xymatrix{
\overline M_g\ar[d]_{q} \ar[r]^{\bar g}&  \ar[d]^{p}  \overline M\\
M \ar[r]^{g}& M \\
}
$$
In particular, we have
\begin{equation}\label{eq:deg}
\deg(g)\deg(q)=\deg(p)\deg(\bar g).
\end{equation}
Clearly,
\begin{equation}\label{eq:index}
[\pi_1(M):q_*(\pi_1(\overline M_g))]\leq[\pi_1(M):p_*(\pi_1(\overline M))]<\infty, 
\end{equation}
which means that  there are finitely many possibilities for the homotopy type of the covering $\overline M_g$, since $\pi_1(M)$ contains only finitely many subgroups of a fixed index. Now, $\deg(\bar g)$ can take only finitely many values because $I(\overline M)\neq0$. By \eqref{eq:deg}, we deduce that there are only finitely many possibilities for $\deg(g)$, which means that $\deg(g)=\pm1$. This implies that $g$ is $\pi_1$-surjective, and thus $\varphi$ is surjective (here, $\varphi=g_*$).

Note that, since our space $M$ is aspherical and $\pi_1(M)$ is Hopfian (even more, residually finite by Theorem \ref{t:He}) the above shows moreover that $g$ is a homotopy equivalence and that $\overline M_g$ is homotopy equivalent to $\overline M$. Indeed, we have $\deg(p)=\deg(q)$ (because \eqref{eq:index} is an equality) and $\bar g$ is a $\pi_1$-isomorphism (it is surjective because $\deg(g)=\pm 1$ and injective since $q_*$ and $g_*$ ar injective). Hence, more generally, we obtain the following positive answer to the Strong version of the Hopf Problem (see Problem \ref{p:sHopf}):

\begin{theorem}[Theorem \ref{t:hopf}]\label{t:hopfagain}
Let $M$ be a closed oriented connected aspherical manifold with Hopfian fundamental group, which has a finite covering $p\colon\overline M\to M$ with $I(\overline M)\neq0$, where $I$ is a functorial numerical invariant. Then any self-map $g\colon M\to M$ of non-zero degree is a homotopy equivalence. In particular, $g$ lifts to a homotopy equivalence $\bar g\colon  \overline M_g\to \overline M$, where $\overline M_g$ is the finite covering of $M$ corresponding to $g_*^{-1}p_*(\pi_1(\overline M))$.
\end{theorem}

\begin{remark}
We point out that the above argument about the $\pi_1$-surjectivity of $g$ does not use that $\pi_1(M)$ is Hopfian. In fact, \cite[Lemma 3.1]{DW} shows that for any manifold $W$ and any map $g\colon W\to M$, there any only finitely many possibilities for $\deg(g)$. We also note that no example of a closed oriented aspherical manifold with non-Hopfian fundamental group is known.
\end{remark}

\begin{remark}
An anonymous referee informed us that Wang-Wu \cite[Theorem 8.9]{WaWu} also showed that the fundamental group of a closed irreducible $3$-manifold $M$ is co-Hopfian if and only if $M$ is not covered by the product of a surface with the circle or a torus bundle over the circle (see also~\cite{GLW}).
\end{remark}

Finally, we are left with the case where $M$ is not prime. Note that the prime summands of $M$ can be any closed $3$-manifolds with non-trivial fundamental group, excluding the case $M=\mathbb RP^3\#\mathbb RP^3$, since $M$ is not finitely covered by $S^2\times S^1$ (see Table \ref{table:3geom}). Thus, the fundamental group of $M=M_1\#M_2$ is given by 
\[
\pi_1(M)=\pi_1(M_1)\ast\pi_1(M_2), 
\]
where $|\pi_1(M_i)|>2$ for at least one of the two $M_i$ (clearly, such a group is never co-Hopfian). By~\cite[Theorem 4.15 (ii)]{Ka}, the first $L^2$-Betti number of $\pi_1(M)$ is given by
\[
b_1^{(2)}(\pi_1(M))=1-\frac{1}{|\pi_1(M_1)|}-\frac{1}{|\pi_1(M_2)|}+b_1^{(2)}(\pi_1(M_1))+b_1^{(2)}(\pi_1(M_2)).
\]
Since $|\pi_1(M_i)|>2$ for at least one of the two $M_i$, we conclude that $b_1^{(2)}(\pi_1(M))\neq0$ and the claim follows by Lemma \ref{l:multiplicative}.
\end{proof}

Corollary~\ref{c:main} for $3$-manifolds follows from Theorems \ref{t:main} and \ref{t:coHopf3}.

\subsection{Proof of Corollary \ref{c:main} in dimension four}

As in the case of $3$-manifolds, we will show that for each of the $4$-manifolds satisfying the hypothesis of 
Corollary \ref{c:main} (ii) every self-map of non-zero degree is a $\pi_1$-isomorphism. Again, we show a stronger group theoretic result:

\begin{theorem}\label{t:coHopf4}
Let $M$ be a closed oriented connected $4$-manifold which is a (possibly trivial) connected sum of geometric $4$-manifolds so that $\pi_1(M)$ is 
not isomorphic to $\mathbb Z_2\ast\mathbb Z_2$ or to the fundamental group of a closed $4$-manifold modeled on one of the geometries $\mathbb X^3\times \bR$, $Sol^4_0$, $Sol^4_1$, $Sol^4_{m\neq n}$ or $Nil^4$. If $\pi_1(M)$ is isomorphic to the fundamental group of a closed aspherical $4$-manifold, then $\pi_1(M)$ is co-Hopfian. Otherwise, $\pi_1(M)$ is finitely co-Hopfian.
\end{theorem}
Note that assuming $\pi_1(M)$ to be isomorphic to the fundamental group of a closed aspherical $4$-manifold, say, $E$, implies that, among the geometric summands that compose $M$, exactly one has non-trivial fundamental group, otherwise $E$ would decompose as a wedge sum. Hence, $\pi_1(M)$ must in fact be isomorphic to the fundamental group of a closed geometric aspherical manifold.
\begin{proof}
Since finite groups are co-Hopfian, we can assume that $\pi_1(M)$ is infinite.

Suppose first that $\pi_1(M)$ is isomorphic to the fundamental group of a geometric, closed aspherical $4$-manifold $E$ and let $\varphi\colon \pi_1(M)\to \pi_1(M)$ be an injective endomorphism. According to Table \ref{table:4geom}, $E$ is modeled on one of the aspherical geometries  
$$\mathbb H^4, \ \mathbb H^2(\mathbb C), \ \mathbb H^2\times\mathbb H^2.$$ The homotopy type of the self-map induced by $\varphi$ on the classifying space $E\simeq B\pi_1(M)$ is given by $f:=B\varphi\colon E\to E$ and has finite degree $\deg(f)=\pm[\pi_1(M):\varphi(\pi_1(M))]$ (see~\cite[Corollary 7.3]{NeEndo}). For each of the above three geometries, the simplicial volume of $E$ does not vanish~\cite{Gr,Bu}, hence Lemma \ref{l:functorial} tells us that $f$ is $\pi_1$-surjective as well, i.e., $\varphi$ is surjective. 
Note that here one can use as well that $\pi_1(M)$ has non-zero Euler characteristic or non-zero second $L^2$-Betti number~\cite[Theorem 4.3]{Lu94}. Thus $\pi_1(M)$ is finitely co-Hopfian by Lemma \ref{l:multiplicative}. Now, if $\varphi\colon\pi_1(M)\to\pi_1(M)$ is injective, then $\varphi(\pi_1(M))$ is isomorphic to $\pi_1(M)$, thus they have the same cohomological dimension. A theorem of Strebel~\cite{St} then tells us that $[\pi_1(M):\varphi(\pi_1(M))]<\infty$ and we conclude that $\varphi$ is an isomorphism, that is, $\pi_1(M)$ is co-Hopfian. (We remark that Strebel's result~\cite{St} was used above as well, in the proof that $\deg(f)=\pm[\pi_1(M):\varphi(\pi_1(M))]$ \cite[Corollary 7.3]{NeEndo}.)

The remaining cases according to Table \ref{table:4geom} are when $\pi_1(M)$ is isomorphic to the fundamental group of a $4$-manifold modeled on $S^2\times\mathbb H^2$ or $\pi_1(M)$ is a non-trivial free product. Assume first that $\pi_1(M)$ is the fundamental group of a $4$-manifold modeled on $S^2\times\mathbb H^2$. According to Table \ref{table:4geomvirtual} each $S^2\times\mathbb H^2$ manifold is finitely covered by an $S^2$-bundle over a hyperbolic surface $\Sigma$. This bundle has fundamental group isomorphic to $\pi_1(\Sigma)$ by the homotopy long-exact sequence, thus $\pi_1(M)$ contains $\pi_1(\Sigma)$ as a subgroup of finite index. In particular, it has non-zero Euler characteristic (or first $L^2$-Betti number). Hence $\pi_1(M)$ is finitely co-Hopfian by Lemma \ref{l:multiplicative}.

Finally, suppose  $M=M_1\#M_2$, where both $\pi_1(M_i)$ are not trivial. 
Since  
$\pi_1(M)$ is not isomorphic to $\mathbb Z_2\ast\mathbb Z_2$, we have that $|\pi_1(M_i)|>2$ for at least one $M_i$. By~\cite[Theorem 4.15 (ii)]{Ka}, the first $L^2$-Betti number of $\pi_1(M)$ is given by
\[
b_1^{(2)}(\pi_1(M))=1-\frac{1}{|\pi_1(M_1)|}-\frac{1}{|\pi_1(M_2)|}+b_1^{(2)}(\pi_1(M_1))+b_1^{(2)}(\pi_1(M_2)).
\]
We conclude that $b_1^{(2)}(\pi_1(M))\neq0$ and the claim follows by Lemma \ref{l:multiplicative}.
\end{proof}

Corollary~\ref{c:main} for $4$-manifolds follows from Theorems \ref{t:main} and \ref{t:coHopf4}: Let $f\colon M\to M$ be a map of non-zero degree, where $M$ fulfills the hypothesis of Corollary~\ref{c:main} (ii). By Theorem \ref{t:main} (ii), $f$ is $\pi_1$-injective, hence to conclude, it suffices to show that $\pi_1(M)$ satisfies the hypothesis of Theorem \ref{t:coHopf4}. Since we are excluding by assumption that $N$ is modelled on $Sol^4_0$, $Sol^4_1$, $Sol^4_{m\neq n}$, $Nil^4$, $\mathbb X^3\times \bR$ and $S^2\times\mathbb X^2$, we only need to ensure that the fundamental group of $M$ is not isomorphic to $\mathbb Z_2\ast\mathbb Z_2$. No closed orientable geometric $4$-manifold has this fundamental group (from Table \ref{table:4geomvirtual}, the only candidate geometry giving an infinite virtually cyclic group would be $S^3\times \bR$), thus, the only case left is when $N$ is a non-trivial connected sum $N = N_1\#N_2$, where both $N_1$ and $N_2$ have fundamental group of order two. But this forces them to be quotients of $S^2\times S^2$. Now, $S^2\times S^2$ has only two such non-trivial orientable geometric quotients $Q$, $Q'$ (see~\cite[Sections 12.2 and 12.3]{Hillman}).
Since we have excluded $Q\#Q$, $Q\#Q'$ and $Q'\#Q'$ we are done.

\medskip

This completes the proof of Corollary~\ref{c:main}.

\end{document}